\documentclass[12pt]{article}
\usepackage{geometry}
\usepackage{amsmath,amsthm,amscd,amssymb,latexsym,upref,enumerate,bbm,bm,lscape,epsfig,nicefrac,hyperref,amsfonts, color,graphicx,xcolor}
\usepackage{mathtools}

\usepackage[numbers]{natbib}

\newtheorem{info}{}
\newtheorem{theorem}[info]{Theorem}
\newtheorem{lem}[info]{Lemma}
\newtheorem{prop}[info]{Proposition}
\newtheorem{proposition}[info]{Proposition}
\newtheorem{corr}[info]{Corollary}
\newtheorem{note}[info]{Note}
\newtheorem{defin}[info]{Definition}

\newtheorem{remark}[info]{Remark}

\numberwithin{info}{section}
\renewcommand{\[}{\begin{equation}}
\renewcommand{\]}{\end{equation}}
\makeatletter
\g@addto@macro\normalsize{%
  \setlength\abovedisplayskip{5pt}
  \setlength\belowdisplayskip{5pt}
  \setlength\abovedisplayshortskip{4pt}
  \setlength\belowdisplayshortskip{4pt}}
\makeatother

\usepackage [english]{babel}
\usepackage [autostyle, english = american]{csquotes}
\MakeOuterQuote{"}

\newcommand{\C}{\mathbb{C}}
\newcommand{\R}{\mathbb{R}}

\renewcommand{\P}{\mathbb{P}}
\renewcommand{\cal}{\mathcal}

\newcommand{\del}{\nabla}

\newcommand{\lam}{\lambda}
\newcommand{\Lam}{\Lambda}

\newcommand{\sq}{\sqrt}
\newcommand{\sign}{\mathrm{Sign}}

\newcommand{\<}{\langle}
\renewcommand{\>}{\rangle}
\newcommand{\supp}{\mathrm{Supp}}
\renewcommand{\sq}{\sqrt}

\newcommand{\Sum}{\mathrm{\Sum}}

\newcommand{\To}{\Rightarrow}
\renewcommand{\Im}{\mathrm{Im}}

\newcommand{\loc}{\mathrm{loc}}
\newcommand{\ind}{\mathrm{ind}}
\renewcommand{\Im}{\mathrm{Im}}

\newcommand{\Crit}{\mathrm{Crit}}

\renewcommand{\H}{\mathbb{H}}

\newcommand{\TW}{\mathrm{TW}}

\begin{document}
\title{Critical Fluctuations for the Spherical Sherrington-Kirkpatrick Model in an External Field}
\author{Pax Kivimae}
\maketitle 

\begin{abstract}
We prove the existence of a critical regime for the fluctuations of the ground-state energy of the spherical Sherrington-Kirkpatrick model in an external field, confirming predictions given in \cite{BaikLee,FD}. We also establish a critical regime for the fluctuations in a model with a critical Ferromagnetic interaction term, producing a three-parameter family of distributions generalizing the two-parameter family given in \cite{BV}. These results are both established in the generality of a $\beta$-ensemble analogue of the spherical Sherrington-Kirkpatrick model, which subsumes the complex and quarternionic generalizations. 
\end{abstract}

\section{Introduction}
Let $A_N=(A_{ij})_{i,j=1}^{N}$ be a real symmetric matrix sampled from the (unnormalized) $N$-by-$N$ Gaussian Orthogonal Ensemble (GOE) and let $h\in \R$. We define the $2$-spin spherical Sherrington-Kirkpatrick model with external field $h$ to be:
\[H_{N,h}(\sigma)=\frac{1}{2\sq{N}}\<A_N\sigma, \sigma\>+h\sum_{i=1}^{N}\sigma_i\label{ssk}\]
 for spin variables $\sigma \in S^{N-1}$ where $S^{N-1}:=\{\sigma\in \R^N:\<\sigma,\sigma\>=N\}$ and $\<u,v\>=\sum_{i=1}^Nu_iv_i$ denotes the Euclidean scalar product. We now fix the notation:
 \[E_{N,h}:=\frac{1}{N}\sup_{\sigma\in S^{N-1}}H_{N,h}(\sigma)\label{optssk}.\]
 We will be interested in understanding the asymptotic behavior of this statistic as a function of $h$.
 
 This model was first introduced in \cite{KTJ}, and has been studied extensively in both the mathematics and physics literature ever since (see \cite{CS,PanchenkoTalagrand,Talagrand} and the references therein). In particular, it has been shown that $E_{N,h}$ converges a.s. to a deterministic value: $\sq{1+h^2}$. For a proof of this fact, and further background on the history of this problem, see \cite{ChenSen}.

The limiting fluctuations of $E_{N,h}$ have been studied in the case of fixed $h$.  In particular, it was shown in \cite{TWO} that when $h=0$, we have that $N^{2/3}(E_{N,0}-1)$ converges in law to $\TW_1$, where $\TW_1$ denotes the GOE Tracy-Widom distribution, and that $\sq{N}(E_{N,h}-\sq{1-h^2})$ converges to a Gaussian limit law otherwise \cite{ChenSen}. On the otherhand, it been conjectured in \cite{BaikLee,FD}, that when $h_N$ is allowed to depend on $N$ in such a way that $N^{1/6}h_N\to h$, the limiting fluctuations of $E_{N,h_N}$ should be described by an interpolating family of distributions depending only on $h$. We answer this question in the affirmative (see Theorem \ref{th1up}). In particular, we produce a family of distributions $\TW_1^h$, such that $\TW_1^0=TW_1$, which describe the limiting law of $-N^{2/3}(1+\frac{h_N^2}{2}-E_{N,h_N})$. We also prove a stronger result characterizing the joint-distribution of the critical values of $H_N$ of large index. In particular, let us denote by $\Crit_{k,k'}(f)$ the set of critical values of a function $f$, which possess an underlying critical point with index in $[k,k']$. We will then construct, for each $k\ge 0$, an a.s finite and nonempty, point process $\Lam_{1}^{h,k}$, such that $\inf \Lam_1^{h,0}=-\TW_1^h$, such that $\Lam_{1}^{h,k}$ describes the limiting behavior of $\Crit_{N-k,N}(H_{N,h})$.

We comment on our method of proof.  The method of \cite{ChenSen} in the case of $h\neq 0$ requires tools specific to the case of Gaussian central limit theorems. In addition, the method of \cite{TWO} relies heavily on the fact that when $h=0$, the problem only depends on the spectrum of $A_N$, which allows the application of powerful tools from the theory of orthogonal polynomials. Instead our approach is based on the method by which \cite{RRV} are also able to treat the $h=0$ case. More specifically, our methods are based on the "stochastic operator approach to random matrix theory", pioneered by \cite{EdelmanSutton}.

Let us fix a choice of $\beta>0$. We recall the $\beta$-Hermite Ensemble defined by \cite{DumiEdel} by:
\[A_N^{\beta}=\frac{1}{\sq{\beta}}\begin{bmatrix}\sqrt{2}\,g_1& \chi_{\beta(n-1)}
\\ \chi_{\beta(n-1)} &  \sqrt{2}\,g_2 & \chi_{\beta(n-2)}
\\ & \chi_{\beta(n-2)} & \sqrt{2}\,g_3 & \ddots
\\ & & \ddots & \ddots & \chi_\beta
\\ & & & \chi_{\beta} &  \sqrt{2}\,g_n
\end{bmatrix}\label{matrixdef}.\]
Here $g_i$ are independent standard Gaussians random variables and $\chi_{\beta(i-1)}$ are independent chi random variables with parameter $\beta(i-1)$. 

This family of matrix ensembles was introduced by \cite{DumiEdel} as a matrix model for the eigenvalue density given by the Coulomb gas model at inverse temperature $\beta$:
\[\frac{1}{Z_{\beta,N}}\prod_{i<j}|\lam_i-\lam_j|^\beta e^{-\beta N\sum_{i=1}^{N}\lambda_i^2/4}.\]
Here $Z_{\beta,N}$ is a normalizing constant. When $\beta=1$ this density coincides with the density of the eigenvalues of $A_N$, allowing them to reduce the study of the eigenvalues of $A_N$ to the study of the eigenvalues of $A_N^1$. Infact, more can be said in this case. One may apply a.s. apply Householder's algorithm to $A_N$, to produce an orthogonal matrix $H(A_N)$, such that $H(A_N)e_1=e_1$, and such that conjugation of $A_N$ by $H(A_N)$ has a tridiagonal form. It is then shown in \cite{DumiEdel} that the law of tridiagonal matrices given by this operation has law given by $A_N^1$. Motivated by this, we define the $\beta$-spherical Sherrington Kirkpatrick model in external field $h$ to be:
\[H_{N,h}^\beta(\sigma)=\frac{1}{2\sq{N}}\<\sigma,A_N^\beta\sigma\>+h\sq{N}\sigma_1.\label{ssk-beta}\]
where we have $\sigma\in S^{N-1}$ as before. Additionally, we define, as before:
\[E_{N,h}^\beta=\frac{1}{N}\sup_{\sigma \in S^{N-1}}H_{N,h}^\beta(\sigma).\]
We note that by the above application of Householder's algorithm, and rotation invariance of the $\mathrm{GOE}$ ensemble, it is clear that the distribution of the critical values of $H_{N,h}^1$ are identical to that of $H_{N,h}$. By an application of Lagrange Duality (see Section \ref{LagrangeSection}) we have that a.s.:
\[E_{N,h}^{\beta}=\inf_{\lam>\lam_{1,\beta}^N}\frac{1}{2}[\lam-h^2\<(\frac{1}{\sq{N}}A_N^{\beta}-\lam)^{-1}e_1,e_1\>]\label{resolvantintroduction}\]
where here $\lam_{1,\beta}^N$ is the largest eigenvalue of $A_N^{\beta}/\sq{N}$. Similar remarks can be made for other critical points of $H_{N,h}^{\beta}$.

Our method of understanding $E_{N,h}^1$ is to understand the right hand-side of (\ref{resolvantintroduction}). The limiting behavior of $A_{N}^\beta/\sq{N}$ around $\lam_{1,\beta}^N$ is understood in \cite{RRV} in terms of a limiting stochastic-operator. In particular, they consider the $\beta$-Stochastic Airy Operator on $L^2(\R_+)$ heuristically given by:
\[\cal{A}_{\beta}=-\frac{d^2}{dx^2}+x+\frac{2}{\sq{\beta}}B_x'\]
where $B_x$ denotes a standard Brownian motion. It is shown in \cite{RRV} that the Diriclet eigenvalue problem for this operator is well-defined, and possess a discrete, bounded-below set of solutions, which we denote $\sigma_D(\cal{A}_\beta)=\{\lam_{1,\beta}\le \lam_{2,\beta}\le \dots\}$. Let us denote the following rescaled matrix $B_N^{\beta}=N^{2/3}(2-A_N^{\beta}/\sq{N})$. They show that the lowest eigenvalue of $B_N^{\beta}$ converges in law to $\lam_{1,\beta}$, and similarly for other eigenvalues. Our result will be to show that this understanding extends to entry of the resolvent present in (\ref{resolvantintroduction}).

In particular in Section \ref{prelimsection} we show that a.s., there is, for each $\lam\notin\sigma_D(\cal{A}_\beta)$, a differential function $\varphi_{\lam,\beta}:[0,\infty)\to \R$, such that $\cal{A}_\beta\varphi_{\lam,\beta}=\lam \varphi_{\lam,\beta}$ and such that $\varphi_{\lam,\beta}(0)=1$. We will show (see Proposition \ref{MainConvergenceTheorem}) that for each $\lam\notin \sigma_D(\cal{A}_\beta)$ that:
\[N\<(B_N^{\beta}-\lam)^{-1}e_1,e_1\>-\frac{N^{1/3}}{2}\To \varphi_{\lam,\beta}'(0)\]
in law. This convergence requires new methods, as results of \cite{RRV,BV} on the convergence of $B_N^{\beta}$ to $\cal{A}_\beta$ have only demonstrated $L^2$-convergence of derivatives. In view of this convergence, we define:
\[-\TW_{\beta}^h=\sup_{\lam<\lam_{1,\beta}}\frac{1}{2}(\lam-h^2\varphi_{\lam,\beta}'(0))\]
It is clear that $\TW_{\beta}^0=\TW_{\beta}$, where $\TW_{\beta}$ is the distribution of \cite{BV}. We also define, for $h\neq 0$,:
\[V_{\beta}^{h}=\{\lam\in \R:\|\varphi_{\lam,\beta}\|^2=h^{-2}\},\]
where $\|.\|$ denotes the $L^2(\R_+)$-norm, and $V_{\beta}^0=\sigma_D(\cal{A}_\beta)$. In either case, we let:
\[V_{\beta}^{h,k}=\{\lam\in V_{\beta}^h:\lam\in (-\infty, \lam_{k+1,\beta}]\text{ or } \lam\in (\lam_{k+1,\beta},\lam_{k+2,\beta}] \text{ and } \frac{d^2}{d\lam^2}\varphi_{\lam,\beta}'(0)<0\},\]
\[\Lambda_{\beta}^{h,k}=\{\frac{1}{2}(\lam-h^2\varphi_{\lam,\beta}'(0)):\lam\in V_{\beta}^{h,k}\}.\]
We show below that $-\TW_{\beta}^h=\inf V_{\beta}^{h,0}$. Our first main result is the following:
\begin{theorem}
Let $\beta>0$ and let $h_N$ be a sequence such that $N^{1/6}h_N\to h$ for some $h$. Then for $k\ge 0$ we have that:
\[N^{2/3}(1+\frac{1}{2}h_N^2-\Crit_{N-k,N}(H_{N,h}^{\beta}))\To \Lam_{\beta}^{h,k}\] in law with respect to the Hausdorff metric.

In particular, we have that:
\[N^{2/3}(1+\frac{1}{2}h_N^2-E_{N,h_N}^{\beta})\To -\TW_{\beta}^h\]
in law.
\label{th1up}
\end{theorem}

\begin{remark}
If $A_N$ in (\ref{ssk}) is replaced with a \textup{GUE (GSE)} matrix, and the state-space is replaced with $\C S^n$ \textup{(}$\H S^n$\textup{)} then Theorem \ref{th1up} with $\beta=2\; (4)$, respectively, establishes a similar theorem on the limiting fluctuations on the ground-state energy of these models. This follows from the same analysis and a more generalized application of Householder's algorithm as in \cite{DumiEdel}.
\end{remark}

Our next result will be a similar transition result for $H_{N}^\beta$ with the addition of a Curie-Weiss term. That is, fixing $\mu$, $h$ and $\beta$ real parameters, with $\beta>0$, we will denote the $\beta$-spherical Sherrington Kirkpatrick model with Curie-Weiss interaction strength $\mu$, and external field $h$, as:
\[H_{N,\mu,h}^{\beta}(\sigma)=H_{N,h}^{\beta}+\frac{\mu}{2}\sigma_1^2\]
where $\sigma\in S^{N-1}$. We will denote:
\[E_{N,\mu,h}^{\beta}=\frac{1}{N}\sup_{\sigma\in S^{N-1}}H_{N,\mu,h}^{\beta}(\sigma).\]
As before, when $\beta=1$, this model has critical values equivalent to (\ref{ssk}) with Curie-Weiss term considered in \cite{BaikLee,KTJ}.

In absence of an external field, this model was studied in the classical cases of $\beta=2$ by \cite{DesFor,Peche} who established an analogue of the Baik-Ben Arous-P\'ech\'e theorem \cite{BPP}. Namely, it is shown that for $\mu<1$, the subcritical regime, the fluctuations of the maximum are of order $N^{-2/3}$, and follow $\TW_2$, the GUE Tracy-Widom Law. On the other hand, when $\mu>1$, the maximum fluctuates of order $N^{-1/2}$, and follows a Gaussian law. The critical regime, where $N^{1/3}\mu_N-1\to w$, was studied by \cite{BV}, and moreover, was done in the case of a general $\beta$. They show that the highest eigenvalue has fluctuates of order $N^{-2/3}$, but with an exotic law, denoted as $\TW_{\beta,w}$. In the case of $\beta=2$, this law can exactly be identified with the critical interpolating law in the Baik-Ben Arous-P\'ech\'e theorem transition (see Theorem 1.2 of \cite{BPP}).

We generalize this story by identifying a joint-critical regime in $\mu$ and $h$. In particular, we now introduce a three parameter family $\TW_{\beta,w}^{h}$, such that $\TW_{\beta,w}^{0}=\TW_{\beta,w}$. To begin, we recall the construction of $\TW_{\beta,w}$. It is shown in \cite{BV} that a.s. the eigenvalue problem of $\cal{A}_\beta$ on $L^2(\R_+)$ with $w$-Robinson boundary conditions (i.e. $w\phi(0)=\phi'(0)$) is well-defined, and possess a discrete, bounded-below set of solutions. We denote this sequence as $\sigma_w(\cal{A}_\beta)=\{\lam_{1,\beta}^w<\lam_{2,\beta}^\beta<\dots\}$.

We show that for each $w$, there is a.s, for each $\lam\notin \sigma_w(\cal{A}_\beta)$, a differential function  $\varphi_{\lam,\beta}^w:[0,\infty)\to \R$, such that $\cal{A}_\beta\varphi_{\lam,\beta}^w=\lam \varphi_{\lam,\beta}^w$, and  $w\varphi_{\lam,\beta}^w(0)+1=(\varphi_{\lam,\beta}^w)'(0)$. We now define:
\[-\TW_{\beta,w}^h=\sup_{\lam<\lam_1^\beta}\frac{1}{2}(\lam-h^2\varphi_{\lam,\beta}^w(0))\]
We also define, for $h\neq 0$:
\[V_{\beta,w}^{h}=\{\lam\in \R:\|\varphi_{\lam,\beta}^w\|^2=h^{-2}\}\]
and $V_{\beta,w}^0=\sigma_w(\cal{A}_\beta)$. In either case, we let:
\[V_{\beta,w}^{h,k}=\{\lam\in V_{\beta,w}^h:\lam\in (-\infty, \lam_{k+1,\beta}^w]\text{ or } \lam\in (\lam_{k+1,\beta}^w,\lam_{k+2,\beta}^w] \text{  and } \frac{d^2}{d\lam^2}\varphi_{\lam,\beta}^{w}(0)<0\}\]
\[\Lambda_{\beta,w}^{h,k}=\{\frac{1}{2}(\lam-h^2\varphi_{\lam,\beta}^w(0)):\lam\in V_{\beta,w}^{h,k}\}\]
As before, we show below that $-\TW_{\beta,w}^h=\inf V_{\beta,w}^{h,0}$. We may now state our second main result:
\begin{theorem}
Let $\beta>0$ and let $h_N$ be such that $N^{1/2}h_N\to h$ for some $h$. Let $\mu_N$ be such that $N^{1/2}(1-\mu_N)\to w$ for some $w$. Then for $k\ge 0$ we have that:
\[N^{2/3}(2-\Crit_{N-k,N}(H_{N,\mu_N,h_N}^\beta))\To \Lam_{\beta,w,h}^k\] in law with respect to the Hausdorff metric.\\
In particular, we have that:
\[N^{2/3}(2-E_{N,\mu_N,h_N}^{\beta})\To -\TW_{\beta,w}^{h}\]
in law.
\label{th1cw}
\end{theorem}

The paper is organized as follows. In Section \ref{LagrangeSection} we review a classical lemma on the behavior of the Lagrange-dual optimization problem, which will prove to be a useful reformation for asymptotic analysis. Section \ref{resultsection} contains the proof of our main results and is laid out in the following way. In Section \ref{reviewsubsection} we review a family of (spiked) tridiagonal matrix ensembles of introduced \cite{BV}, which in particular contain $B_N^{\beta}$. We then introduce in Section \ref{InhomoSubsection} a quadratic function, similar to (\ref{ssk}), for each member of this family, and characterize their low-lying critical points in Proposition \ref{KeyTheorem}, proven later in the Section. In Section \ref{ProofSection} we show that Proposition \ref{KeyTheorem} yields Theorem \ref{th1up} and Theorem \ref{th1cw} above. In Section \ref{Deterministic}, we reduce the proof of Proposition \ref{KeyTheorem} to the proof of a deterministic statement, namely Proposition \ref{MainConvergenceTheorem}. Proposition \ref{MainConvergenceTheorem} shows that the rescaled resultant of a (spiked) tridiagonal matrix around the edge converges uniformly to a function constructed from certain eigenfunctions of the continuum operator, up to a diverging constant, and is proven in the remainder of the section. Finally, in Section \ref{prelimsection} we develop the theory of a certain class of Stochastic Operators (introduced in \cite{RRV}), which contain the Stochastic Airy Operator as a special case. In particular, we show that the eigenvalue problems introduced in \cite{BV,RRV} are realized by self-adjoint operators. In addition, we show the existence of a special family of eigenfunctions (see Proposition \ref{WeylSolutions}) which play a key role in the definitions of Section \ref{resultsection}.

\subsection{Notation}
We will use the notation $L^p$ to denote $L^p(\R_+)$, and similarly for other function spaces. We will additionally employ the notation $\R^*=\R\cup \{\infty\}$.

\subsection{Acknowledgements}

The author would like to thank Antonio Auffinger for proposing this project and advising him throughout its completion. Additionally, the author would also like to thank Christian Gorski and Julian Gold for several helpful conversations during the duration of this project. The author would like to thank Jinho Baik for letting them know of \cite{baik2020}, which considers, among other things, the positive temperature version of the problem considered here. The work of the author was supported in part by  NSF RTG-grant 1502632 and NSF CAREER 1653552.

\section{Lagrange Duality for Quadratic Optimization \label{LagrangeSection}}

In this section we recall a form of Lagrange duality coming from optimization theory. These results are recalled for notational clarity, and are effectively contained in \cite{Forc}.

For this section fix a symmetric $N$-by-$N$ matrix $H$, and $v\in \R^N$. We will assume that $v$ is neither an eigenvector of $H$, nor zero, and that the eigenvalues of $H$ are distinct. Let us define a quadratic function by \[L(\sigma)=\frac{1}{2}\<H\sigma,\sigma\>+\<\sigma, v\>\]
where $\sigma\in S^{N-1}$. We are interested in the critical points of this function.

By use of Lagrange's Method, one obtains the following critical points equations:
\[\begin{cases}
H\sigma+v=\lam \sigma.\\ \<\sigma,\sigma\>=N.
\end{cases}
\label{criteqn}\]
Using our assumption on $v$, we may rewrite the first of these as: \[\sigma=-(H-\lam )^{-1}v.\] We will denote the right hand side of this equation as $\sigma_{\lam}$. Then, we see from (\ref{criteqn}) that the critical points of $L_H$ are precisely the choices of $\lambda$ such that: \[\<\sigma_{\lam},\sigma_\lam\>=N.\] Now substituting $\sigma_{\lam}$ into our expression for $L$, we obtain:
\[L(\sigma_\lambda)=\frac{1}{2}\<H(H-\lam)^{-1}v,(H-\lam)^{-1}v\>-\<(H-\lam)^{-1}v\>=\frac{1}{2}[N\lam-\<(H-\lambda)^{-1}v,v\>].\]
We will denote: \[J(\lambda):=L(\sigma_{\lam}). \label{Jdef}\] It turns out that the critical point structure of $L$ can be completely recovered from that of $J$. To state this more precisely, let us denote the ordered eigenvalues of $H$ as $(\mu_i)_{i=1}^N$, such that $\mu_i\ge \mu_{i+1}$.

\begin{theorem}
Let $\Lambda:=\{\lambda_i\}_{i=1}^k$ denote the critical points of $J(\lam)$ with $\lam_i>\lam_{i+1}$. Then $J(\lam_i)>J(\lam_{i+1})$ and additionally we have that $\{\sigma_{\lambda_i}\}_{i=1}^{k}$ are the critical points of $L$.\\
Moreover, we have that: \[\begin{cases}\# \Lambda\cap (\mu_{i},\mu_{i+1})\le 2;\; 1\le i<N\\ \#
\Lambda\cap(-\infty,\mu_n)=1\\ \# \Lambda\cap(\mu_1,\infty)=1.\end{cases}\label{solcount}\]\\
 For $\lam\in \Lambda \cap (\mu_1,\infty)$, $\sigma_{\lam}$ is the unique global maximum, and similarly for $\lam\in \Lambda \cap(-\infty,\mu_N)$, $\sigma_\lam$ is the unique global minimum. If $\cal{J}\cap (\mu_{i+1},\mu_i)=\{\lam_0,\lam_1\}$ with $\lam_0<\lam_1$, then $\sigma_{\lam_j}$ is of index $N-(i-j)$. If $\Lambda \cap(\mu_{i+1},\mu_i)=\{\lam\}$, then $\sigma_\lam$ is of index $N-(i-1)$.
\label{duality-theorem}
\end{theorem}

\begin{proof}
All the statements follow from Theorem 4.1 of \cite{Forc}, except for the number of solutions in $(\mu_{i+1},\mu_i)$ for $1\le i<N$, and the index statements of the associated critical points.

To establish (\ref{solcount}), denote $v_i$ as a normalized eigenvector corresponding to $\mu_i$. We have then that:
\[\<\sigma_\lam,\sigma_\lam\>=\sum_{i=1}^{N}\frac{\<v_i,v\>^2}{(\mu_i-\lam)^2}.\label{CritPointEqn}\]
Thus the critical point equations are:
\[\sum_{i=1}^{N}\frac{\<v_i,v\>^2}{(\mu_i-\lam)^2}=N.\]
The function on the left is positive, convex, with positive poles at $\mu_i$. This proves (\ref{solcount}).\\
For the index statement, note that if $\sigma_\lam$ is a critical point, then:
\[\ind_{S^{N-1}}(\del^2_{\sigma_\lam} L)=\ind (P_{\sigma_\lam}(H-\lam)P_{\sigma_\lam})\]
where here $P_{v}$ denotes the orthogonal projection onto $\{v\}^{\perp}$. We denote this quantity as $f(\lam)$. This function is lower-semicontinuous in $\lam$, by lower-semicontinuity of the index. We show, for $\lam\in [\mu_{i+1},\mu_{i}]$, that $f(\lam)=N-(i-1)$ if $J''(\lam)<0$ and $f(\lam)=N-i$ if $J''(\lam)\le 0$. In view of (\ref{CritPointEqn}) and the properties after, this proves the claim.

To show this, we first note that $f(\mu_i)=N-(i-1)$. Now recall the classical formula, where $A$ is an invertible matrix, and $v$ a vector:
\[\det(P_vAP_v)=\det(A)\<v,A^{-1}v\>.\]
We thus have that:
\[\det(P_{\sigma_\lam}(H-\lam)P_{\sigma_\lam})=-\det(H-\lam)v^t(H-\lam)^{-3}v=J''(\lam)\det(H-\lam).\label{detformula}\]
Now note that $f(\lam)$ may only change at $\lam$ where $\det(P_{\sigma_\lam}(H-\lam)P_{\sigma_\lam})=0$, and thus it may only change at the unique point where $J''(\lam)=0$. Now the claim follows from lower-semicontinuity, and the boundary values.
\end{proof}

We note immediately an important corollary:
\begin{corr}
We have:
\[\sup_{\sigma\in S^{N-1}} L(\sigma)=\inf_{\lambda>\mu_1} J(\lambda).\]
Moreover, if $\lambda$ achieves the infimum on the right, then $\sigma_{\lambda}$ achieves the supremum on the left.
\end{corr}
We also record the following observation:
\begin{note}
We from the proof of Theorem \ref{duality-theorem}, that a vector $\sigma_\lam$ is of index greater than or equal to $N-i$ if and only if either $\lam\ge \mu_{i+1}$ or $\mu_{i+2}<\lam<\mu_{i+1}$ and $J''(\lam)<0$. This criterion will prove useful later.
\label{indexnote}
\end{note}

\section{Results on Spiked Tridiagonal Matrix Ensembles \label{resultsection}}

In this section we will prove a general convergence theorem for the low-lying critical points of a class of inhomogeneous functionals based on the families of spiked tridiagonal matrices introduced in \cite{BV} (see Proposition \ref{MainConvergenceTheorem}). Using this, we prove Theorem \ref{th1up} and \ref{th1cw}. The proof of Proposition \ref{MainConvergenceTheorem} relies on Proposition \ref{KeyTheorem}, which is involved, requiring the usage of a new discrete-to-continuous convergence result given in Section \ref{quasi-section}, and a recursion given in Section \ref{ConvergenceSubsection}.

\subsection{Review of Spiked Tridiagonal Matrix Ensembles \label{reviewsubsection}}

We begin by reviewing the set-up of spiked tridiagonal matrix ensembles as developed in \cite{BV,RRV}. This set-up contains the family (\ref{matrixdef}) (see Section 6 of \cite{RRV}) and thus will be general enough for our purposes.

Fix a sequence $m_N\in \R_+$ such that both $m_N=o(N)$ and $m_N=\omega(1)$. Such a choice of $m_N$ defines an isometric embedding:\
\[\R^N\xhookrightarrow{} L^2,\text{ with } e_i\mapsto I_{[(i-1)/m_N,i/m_N)}\]
where $\R^N$ is endowed with the norm $\|v\|^2=m_N^{-1}\sum_{i=1}^{N}v_i^2$. We will also use the notation $(v,w)=m_N^{-1}\sum_{i=1}^Nv_iw_i$ for $v,w\in \R^N$. We will in-general ignore the distinction between a vector in $\R^N$ and its image in $L^2$. We define translation operator $T_Nf(x)=f(x+m_N^{-1})$, its adjoint $T_N^*(f(x))=f(x-m_N^{-1})I_{[m_N^{-1},\infty)}(x)$, as well as the difference quotient $D_N=m_N(T_N-1)$, as operators on $L^2$. These extend the operators on $\R^N$ given by the left-shift, right-shift, and the discrete derivative, respectively. We will also consider the discrete delta-function at the origin, $m_NE_{11}$, where $E_{ij}$ is the $(i,j)$-th elementary matrix and the orthogonal projection of $L^2\to \R^N$, which we will denote $P_N$. Lastly, for a vector $v$, we consider the notation $v_\times$ to denote term-wise multiplication by $v$.

Now let $(y^N_{i;j})_{j=0,\dots,N}$ for $i=1,2$ be a pair of discrete-time real-valued random processes with $y^N_{i;0}=0$, and $w_N\in \R$ a sequence of random variables. We will denote the image of the sequence $(y^N_{i;j})$ in $L^2$ as $y^N_i$. We define:\[H_{N,y_N,w_N}=P_N\left(D_N^*D_N+(D_Ny_1^N)_{\times}+(D_Ny_2^N)_{\times}\frac{1}{2}(T_N+T_N^*)+w_Nm_NE_{11}\right)\label{tridiagonalmatrixdef}.\]

We will now make a sequence of assumptions on $y^N_{i;j}$. Let us assume we have a continuous random-process of $\R_+$, $y$, such that $y(0)=0$.\\
\textbf{Assumption 1: (Tightness and Convergence)}
\[\{y^{N}_{i}(t)\}_{t\ge 0}, i=1,2 \text{ are tight in law,} \]
\[y^N_1+y^N_2\To y\]
where both statements are taken with respect to the compact-uniform topology on paths\\

Now we will now assume there exists deterministic, unbounded, nondecreasing continuous functions $\bar{\eta}>0$, $\zeta\ge 1$, which satisfy the following assumption.\\

\textbf{Assumption 2: (Growth and Oscillation Bounds)}\\
There exists $\eta^N_{i;j}\ge 0$, $i=1,2$, $j=1,\dots,N$, and random constants $\kappa_N$ (tight in distribution, and all defined on the same probability space as $y_{i;j}^N$) with the following properties.\\ Define: \[\omega^N_{i;j}:=m_
N^{-1}\sum_{k=0}^{j-1}\eta_{i;k}^N-y^N_{i;j}.\] Then we assume that the following bounds hold for each $N$ a.s.:
\[\bar{\eta}(x)/\kappa_N-\kappa_N\le \eta^N_1(x)+\eta^N_2(x)\le \kappa_N(1+\bar{\eta}(x))\label{discretegrowthbound1},\]
\[\eta^N_2(x)\le 2m_N^2,\]
\[|\omega_{i}^n(\xi)-\omega_{i}^N(x)|\le \kappa_N(1+\bar{\eta}(x)/\zeta(x)).\]
for all $x,\xi\in [0,N/m_N]$ with $|x-\xi|\le 1$.

Finally recall Assumption 3 of \cite{RRV}.\\

\textbf{Assumption 3: (Convergence of Spike)} There exists a constant $w\in \R^*$, such that $w_N\to w$ in probability.\\

We will refer to this model as a $w$-spiked model. This class of models is known to converge in distribution (in the norm-resolvant sense) to the operator $\cal{H}=-\frac{d}{dx^2}+y'$, with $w$-Robinson boundary conditions. See Theorem 9 and Remark 10 of \cite{KRV} to show this result in view of our Proposition \ref{idenificationprop}. See Section \ref{prelimsection} for the rigorous definition of this operator.

We comment on one potentially confusing point. If one wishes to make sense of the operator $\cal{H}$, defined in Section \ref{prelimsection}, one has to make sure various bounds exist. It is not at this point clear that a decomposition of $y$ as in (\ref{decomposition}) holds. On the other hand, an application of Prokhorov's Theorem may be used to show that indeed such a decomposition exists a.s (See the Proof of Theorem 5.1 of \cite{RRV}). This result is also recalled in our Section \ref{Deterministic}, though remarked here for clarity. In particular, we may assume all results of Section \ref{prelimsection} hold for $\cal{H}_{w}$ a.s.

For the remainder of this paper, unless otherwise stated, we will fix a choice of such a family of tridiagonal matrix ensembles. We will commonly use abuse of notation $H_{N,w}:=H_{N,y_N,w_N}$, and notate $\cal{H}$ without making clear the choice of $y$.

\subsection{An Inhomogeneous Problem for $H_{N,w}$ \label{InhomoSubsection}}
In this subsection we state a general result on the critical values of spiked tridiagonal matrix ensembles, which recovers Theorem \ref{th1up} and \ref{th1cw} as a special case. For this subsection, we fix a choice of spiked tridiagonal model $H_{N,w}$. Namely, let us define: \[L_{N,w,h}(\sigma)=\frac{1}{2}(H_{N,w}\sigma,\sigma)-h\sigma_1\label{Ldef}\]
where $\sigma\in S_{N-1}:=\{\sigma\in \R^N:\|\sigma\|^2=1\}$ and $h\in \R$. We remind the reader that $\|\sigma\|^2=\frac{1}{m_N}\sum_{i=1}^{N}\sigma_i^2$, so that the the condition $\|\sigma\|^2=1$ is equivalent to $\<\sigma,\sigma\>=m_N$. We recall the Lagrangian dual-function to (\ref{Ldef}): \[J_{N,w,h}(\lam)=\frac{1}{2}(\lam-h^2(R_{N,w}(\lam)me_1,me_1))\]
where we denote $R_{N,w}(\lam):=(H_{N,w}-\lam)^{-1}$.

We now define a family of stochastic processing that expresses the continuum limit. For this, we recall the results Proposition \ref{WeylSolutions}. For $w\in \R^*$ and $\lam\notin \sigma(\cal{H}_w)$, there is a choice of $\varphi_\lam^w\in L^2$, lying in the domain of the operator $\cal{H}$, such that $\cal{H}\varphi_\lam^w=\lam \varphi_\lam^w$, and such that $w\varphi_\lam^w(0)+1=(\varphi_\lam^w)'(0)$ if $w\in \R$ and such that $\varphi_\lam^\infty(0)=1$ in the infinite case. Given these functions, and a choice of $h$, we define:
\[\cal{J}_{w,h}(\lam)=\frac{1}{2}(\lam-h^2\varphi^{w}_{\lam}(0));\quad w\in \R\]
\[\cal{J}_{\infty,h}(\lam)=\frac{1}{2}(\lam-h^2(\varphi^{\infty}_{\lam})'(0)).\]

Let us denote by $\lam_i$ ($\lam_i^N$) the $i$th eigenvalue of  $\cal{H}_{w}$ ($H_{N,w}$), labeled in increasing order. We define a family of point-process parametrized by $w\in \R^*$. For $h\neq 0$, let:
\[V_{w,h}=\{\lam\in \R:\|\varphi_{\lam}^w\|^2=h^{-2}\}\]
\[\Lam_{w,h}=\{\cal{J}_{w,h}(\lam):\lam\in V_{w,h}\}.\]
We also let $V_{w,0}=\Lam_{w,0}=\{\lam_i\}_{i}$.
We note that when $h\neq 0$, then $V_{w,h}$ ($\Lam_{w,h}$) are the critical points (critical values) of $\cal{J}_{w,h}$, respectively. We also define for $k\ge 0$ and all values of $h$,
\[V_{w,h}^k=\{\lam \in V_{w,h}:\lam\in (-\infty,\lam_{k+1}]\text{ or }\lam\in (\lam_{k+1},\lam_{k+2}] \text{ and } \cal{J}_{w,h}''(\lam)>0\}\]
\[\Lam_{w,h}^k=\{\cal{J}_{w,h}(\lam):\lam\in V_{w,h}^k\}\label{ProcessDef}.\]\\
Note that if $h=0$, we have:
\[V_{w,0}^k=\Lam_{w,0}^k=\{\lam_i\}_{i=1}^{k+1}.\]

We have the following result:
\begin{proposition}
    Let $v^i$ denote a choice of normalized eigenvectors of $H_{N,w}$. Assume that for $N$ sufficiently large, $((v^i_1)_{i=1}^{N},(\lam_i)_{i=1}^{N})\in S^{N-1}\times\{x\in \R^N:x_i\le x_{i+1}\}$ has a continuous law, and that $\frac{1}{2}y^N_{2;1}\neq m_N$ a.s.\\
    Then let $h,h_N\in \R$ be such that $h_N\to h$. Assume that $w\in \R$. Then for any $k\ge 0$ we have:
    \[\Crit_{0,k}(L_{N,w_N,h_N})\To \Lam_{w,h}^k.\]
    If $w=\infty$, then for any $k\ge 0$, we have that:
    \[\Crit_{0,k}(L_{N,w_N,m_Nh_N})-\frac{1}{2}h_N^2w_N\To \Lam_{\infty,h}^k.\]
Here all convergences are in law with respect to the Hausdorff metric.
    \label{KeyTheorem}
\end{proposition}
The proof of this will be postponed to Section \ref{Deterministic}. We comment on the additional assumptions in Proposition \ref{KeyTheorem}. The proof of Proposition \ref{KeyTheorem} follows from a convergence result of $J_{N,w_N,h_N}$ to $\cal{J}_{w,h}$ (See Proposition \ref{MainConvergenceTheorem}), which does not require any of these additional assumptions. The proof of Proposition \ref{KeyTheorem} uses the relation between $J_{N,w,h}$ and $L_{N,w,h}$ given by Theorem \ref{duality-theorem} though, and needs to avoid the possibility that $J_{N,w,h}$ possesses degenerate critical points. These are why we need these additional assumptions.

\subsection{Proof of Theorem \ref{th1up} and Theorem \ref{th1cw}\label{ProofSection}}

In this subsection, we give a proof of Theorem \ref{th1up} and \ref{th1cw} by relating them to the statements of Proposition \ref{KeyTheorem}.

For the purposes of this section, fix $\beta>0$, and us take \[y=\frac{1}{2}x^2+\frac{2}{\sq{\beta}}B_x\] where $B_x$ denotes a standard Brownian motion. For this choice of $y$, we see that $\Lam_{\beta,w}^{h,k}$ coincides with the $\Lam^k_{w,h}$ defined in (\ref{ProcessDef}). We note that in the case that $h=0$, we have that $\Lam_{_\beta,w,0}^k=(\lam_i^w)^{k+1}_{i=1}$, where $\lam_i^w$ is the $i$-th lowest eigenvalue of $\cal{H}_w$. We also note that thus $\TW_{\beta,w}^0=\TW_{\beta,w}$, where $\TW_{\beta,w}$ is the distribution of \cite{BV}. We note that in view of (\ref{Supformula}), we have that:
\[-\TW_{\beta,w}^h=\sup_{\lam<\lam_1^w}\cal{J}_{w,h}(\lam),\label{altdef}\]
as claimed in the introduction.

We now proceed with the proofs of Theorem \ref{th1up} and \ref{th1cw}. To do so, we first must relate $A_N^\beta$ given by (\ref{matrixdef}) to a tridiagonal matrix ensemble. In preparation, set: \[m_N=N^{1/3};\quad B_N:=N^{2/3}(2-\frac{1}{\sq{N}}A_N^\beta).\] A tridiagonal decomposition for $B_N$ is proven in Section 6 of \cite{RRV}. In particular, $B_N$ is of the form of (\ref{tridiagonalmatrixdef}), with:
\[y_{1;k}^N=-N^{-1/6}(2/\beta)^{1/2}\sum_{\ell=1}^kg_\ell;\quad y_{2;k}^N=-N^{-1/6}\sum_{\ell=1}^k2(\sq{N}-\frac{1}{\sq{\beta}}\chi_{\beta(N-\ell)});\quad w_N=m_N \label{triparameters}.\]
They moreover show that Assumption 1 and 2 of Section \ref{prelimsection} are satisfied, and that \[y_1^N+y_2^N\To x^2+\frac{2}{\sq{\beta}}B_x\] in law with respect to the compact-uniform topology. Now note that $B_{N,\mu}:=H_N-N^{2/3}\mu E_{1,1}$ satisfies (\ref{tridiagonalmatrixdef}) with the same $(y_i^N)_{i=1,2}$, and $w_N=N^{1/3}(1-\mu)$. Now take $\mu_N$ to be of the form in either Theorem \ref{th1up} or \ref{th1cw}. We see that $B_{N,\mu_N}$ satisfies the assumptions of of a $w$-spiked ($\infty$-spike) tridiagonal ensemble as in Subsection \ref{reviewsubsection}.

The assumptions on $((v^i_1)_{i=1}^N,(\lam_i)_{i=1}^{N})$ in Proposition \ref{KeyTheorem} follow from Theorem 2.12 of \cite{DumiEdel}. The statement on $y_{2;1}^N=N^{-1/6}(2\sq{N}-\frac{1}{\sq{\beta}}\chi_{\beta(N-1)})$ follows from the continuity of the law of chi random variables. In particular, we may apply Proposition \ref{KeyTheorem} to the ensemble $B_{N,\mu_N}$.

Now let $\sigma\in S^{N-1}$ be a critical point of $H_{N,\mu,h}^{\beta}$. We need to understand the expression:
\[N^{2/3}(1-\frac{1}{N}H_{N,\mu,h}^\beta(\sigma)).\]
Expanding and rearranging powers, we obtain:
\[N^{2/3}(1-\frac{1}{N}H_{N,\mu,h}^\beta(\sigma))=\frac{1}{2}N^{2/3}(2-\frac{1}{N}\<(\frac{1}{\sq{N}}A_N^{\beta}+\mu E_{11})\sigma, \sigma\>)-h\sigma_1N^{1/6}.\]
Recalling that $\<\sigma,\sigma\>=N$, we see that this is equal to:
\[\frac{1}{2N}\<B_{N,\mu}\sigma,\sigma\>-h\sigma_1N^{1/6}.\]
Now we make the substitution $\sigma=N^{1/3}\bar{\sigma}$. This gives $\bar{\sigma}\in S_{N-1}$, and leaves us with:
\[\frac{1}{2}(H_{N,\mu}\bar{\sigma},\bar{\sigma})-(h\sq{N})\bar{\sigma}_1.\]
This is simply $L_{N,w_N,h\sq{N}}(\bar{\sigma})$. 
Running this analysis backwards, we see that this a.s establishes a bijection, $\sigma\leftrightarrow \bar{\sigma}$,  between critical points of index $k$ of $H_{N,\mu_N,h}^\beta$, and critical points of index $N-k$ of $L_{N,w_N,h\sq{N}}$, with the relation: \[N^{2/3}(2-\frac{1}{N}H_{N,\mu_N,h}(\sigma))=L_{N,h\sq{N}}(\bar{\sigma}).\]

Now taking $h_N,h$ as in Theorem \ref{th1up} (\ref{th1cw}), applying Proposition \ref{KeyTheorem} to $h_N\sq{N}$ ($h_NN^{1/6}$) respectively, we are done with the proofs of Theorems \ref{th1up} and \ref{th1cw}.

\subsection{Reduction to the Deterministic Setting\label{Deterministic}}
In this subsection we reduce the proof of Proposition \ref{KeyTheorem} to a deterministic statement. The following analysis is identical to that in the proof of Theorem 5.1 of \cite{RRV} (see also \cite{BV}), and is recalled for the benefit of the reader. First, select any subsequence of $N$. One notes that all the processes $((y_i^N)_{i=1,2},(\int_0^x\eta_i(y)dy)_{i=1,2})$ are tight in distribution, so by Prokhorov's theorem, we may find continuous random processes ($(y_i)_{i=1,2}$,$(\eta^{\dagger}_i)_{i=1,2}$), and a random variable $\kappa$, such that up to passing to a further subsequence:
\[\begin{split}
y_i^N\To y_i;\quad
\int_0 \eta_{N,i}\To \eta^{\dagger}_i;\quad
\kappa_N\To \kappa
\end{split}\]
for $i=1,2$. Here the convergence in the first two equations is in law with respect to the compact-uniform topology on paths. By Skorokhod's representation theorem, we may find realizations of all of these processes on a single probability space, such that the above convergences hold a.s. and in addition $w_N\to  w$ a.s.. We note that (\ref{discretegrowthbound1}) implies a local-Lipshitz bound on $\eta^{\dagger}_i$ for $i=1,2$. Thus there are $(\eta_i)_{i=1,2}$ such that $\eta_i=(\eta^{\dagger}_i)'$ a.e. and such that (\ref{growth1}) holds for $\eta=\eta_1+\eta_2$.

One may now check that $m_N^{-1}\sum_{j=0}^{[x/m_N]}\eta_{i;j}$ converges to $\int_0^x\eta_i$ compact-uniformly. Therefore, we must have  for $i=1,2$, some continuous random process, $\omega_i$, such that $\omega_i^N\To \omega_i$ a.s. in the compact-uniform sense, and such that (\ref{growth2}) hold for $\omega=\omega_1+\omega_2$. 

Once such a subsequence is chosen, some powerful statements can be made about the convergence of $H_{N,w}$ to $\cal{H}_{w}$. The following is noted as Theorem 9 in \cite{KRV}, following directly from the results of \cite{BV,RRV}:

\begin{proposition}\textup{\cite{KRV}}
With such a joint-coupling on a single probability space, $H_{N,w_N}$ converges to $\cal{H}_{w}$ in the norm-resolvent sense a.s.
\label{resolvantconvergence}
\end{proposition}

The proof of Proposition \ref{KeyTheorem} will follow from the following result:
\begin{proposition}
Assume we are in a subsequence of $N$ with joint coupling as above. Assume we have choosen $h_N,h$ as in the assumptions of Proposition \ref{KeyTheorem}. If $w\in \R$, then we a.s. have that: \[J_{N,w_N,h_N}(\lambda)\To \cal{J}_{w,h}(\lam)\] compact-uniformly for $\lam\notin \sigma_w(\cal{H}_{})$.\\
If $w=\infty$, then we a.s. have \[J_{N,w_N,h_N}(\lambda)-\frac{1}{2}h_N^2w_N\To \cal{J}_{\infty,h}(\lam)\]
compact-uniformly for $\lam\notin \sigma_\infty(\cal{H})$.
\label{MainConvergenceTheorem}
\end{proposition}

The proof of this Propositon is postponed to Section \ref{ConvergenceSubsection}. We will see how it implies Proposition \ref{KeyTheorem}.

\begin{proof}[Proof of Proposition \ref{KeyTheorem}]
To establish convergence in law, it suffices to establish that every subsequence has a further subsequence that converges to that limit in law. Thus we see that to prove Proposition \ref{KeyTheorem}, it suffices to a.s convergence of the above quantities once we have passed to a subsequence as above. Thus we will assume we are in such a deterministic subsequence for the remainder of this proof. We will show that all of the convergences of Proposition \ref{KeyTheorem} converges hold a.s.

Let us note that by our assumptions on the distribution of $(\lam_i)_{i=1}^N$, that $H_{N,w}$ a.s has simple eigenvalues. Additionally, as the condition that $e_1$ is not an eigenvector of $H_{N,w}$ is equivalent to the condition that $m_N^2\neq \frac{1}{2}y_{2;1}^Nm_N$, which is also assumed to occur a.s. Thus it suffices to assume both results hold, and in particular, that the results of Section \ref{LagrangeSection} may be applied to $H_{N,w}$.

For convenience, let $h_N'=h_Nm_N$ if $w=\infty$, and $h_N'=h_N$ if $w\in \R$. Fix $k\ge 0$. When $h\neq 0$, let us define:\[V_{N,w,h}=\{\lam:J_{N,w,h}'(\lam)=0\}.\]
When $h=0$, let $V_{N,w,h}=(\lam_i^N)_{i=1}^N$. In either case, define: \[V_{N,h}^{k}=\{\lam\in V_{N,w,h}:\lam\in (-\infty,\lam_{k+1}^N]\text{ or }\lam\in(\lam_{k+1}^N,\lam_{k+2}^N]\text{ and }J_{N,w,h}''(\lam)>0\}.\]
We note that by Remark \ref{indexnote}, we have that:
\[\Crit_{0,k}(L_{N,w,h})=\{J_{N,w,h}(\lam):\lam\in V_{N,w,h}^k\}.\]
In the view of Proposition \ref{MainConvergenceTheorem}, we see that both convergences of Proposition \ref{KeyTheorem} follow from the statement that:
\[V_{N,w,h'}^k\To V_{w,h}^k.\label{Vconv}\]
a.s with respect to the Hausdorff metric.
\begin{remark}
Infact, this shows the stronger statement that $\{(\lam,J_{N,w,h'}(\lam)):\lam\in V_{N,w,h'}^k\}$ converges to $\{(\lam, \cal{J}_{w,h}(\lam):\lam \in V_{w,h}^k\}$ in law. \label{strictremark}
\end{remark}

To show this, note that as $J_{N,w,h}(\lam)$ are real meromorphic functions, we have in addition to compact-uniform convergence that:
\[J_{N,w,h_N'}^{'}(\lam)\To \cal{J}_{w,h}^{'}(\lam)\]
compact-uniformly in $\lam\in \R-\sigma_w(\cal{H}_{})$.

We now dispense with another technicality. Let us denote the event:
\[B_N=\{\text{ there is }\lam \text{ that solves }J'_{N,w,h}(\lam)=0\text{ and }J''_{N,w}(\lam)=0 \}.\]
In view of (\ref{CritPointEqn}), and using our assumption on the continuity of the law of $((q_{i})_{i=1}^N,(\lam_{i})_{i=1}^N)$, it is easy to see that $\P(B_N)=0$. As $(J'_{N,w},J''_{N,w})$ converge to $(\cal{J}'_{w},\cal{J}''_{w})$ a.s in the compact-open topology, we see that:
\[B=\{\text{ there is }\lam \text{ that solves }\cal{J}'_{w}(\lam)=0\text{ and }\cal{J}''_{w}(\lam)=0 \}\]
has probability 0 as well. We neglect both of these subsets. We note that when $h_N'=0$, the desired convergence reduces to convergence of the Eigenvalues, and thus follows from Proposition \ref{resolvantconvergence}. Thus we may assume that $h_N'\neq 0$ for the remainder of this proof. We now recall the following basic lemma on convex functions.
\begin{lem}
Let $F_N:I\to \R$ be a sequence of convex functions on an interval $I=[a,b]$, with $a,b\in\R^*$, and such that $F_N(a)=F_N(b)=\infty$. Assume that $F_N$ converges to a function $F$, with the same properties. Let $x^{*}$ denote the unique infimum of $F$.\\Let $c\in \R^*$ and let $c_N\to c$.
If $c>x^*$, let $y^{\pm}$ be unique point such that $F(y)=c$ and $\sign(F'(y))=\pm 1$. Then for sufficiently large $N$, there exists a unique $y_N^{\pm}$ such that $F_N(y_N)=c_N$ and $\sign(y_N^{\pm})=\pm 1$ such that $y_N\to y$.\\
If $c<x^*$, then $F_N(x)=c_N$ has no solutions for $N$ large.
Lastly, denoting the infimizer of $F$ as $y^*$ and the infimizers of $F_N^*$ as $y_N^*$, we have that $y_N^*\to y^*$
\end{lem}

 One may apply this result to the convex function $\partial_\lam (R_{N,w}(\lam)me_1,me_1)$ on an interval $[\lam_{i}^N,\lam_{i+1}^N]$ and the value $c_N=(1/h_N')^2$. This gives precisely the critical points of $\cal{J}_{N,w,h_N'}$ on $[\lam_i^N,\lam_{i+1}^N]$. We note that while the lemma naively doesn't apply, as these functions are defined on different domains, this can be dealt with a simple reparametrization as $\lam_i^N\to \lam_i$. This yields the following result.

For each $i\ge 1$, let $(\nu_{i,j}^N)_{j=1,2}$, denote, if such points exist, the unique points such that $\nu_{i,j}^N\in (\lam_i^N,\lam_{i+1}^N)$, with $J_{N,w,h_N'}'(\nu_{i,j}^N)=0$, and $\sign(J_{N,w,h_N'}''(\nu_{i,j}^N))=(-1)^j$. By removing $B_N$, we guarantee that these are the only critical points in $(\lam_i^N,\lam_{i+1}^N)$. There are similar points $(\nu_{i,j})_{j=1,2}$, with the same identities for $\cal{J}_{w,h}$. Again, the exclusion of $B$ guarantees these as the unique critical points. For each $i$, there are two possible cases. In the case that $(\nu_{i,j})_{j=1,2}$ exists, we have that $\nu_{i,j}^N$ exists for large $N$, and that such that $\nu_{i,j}^N\to \nu_{i,j}$  for $j=1,2$. In the second case, no solutions to $\cal{J}_{w,h}'=0$ exist on $(\lam_i,\lam_{i+1})$ and no solutions $J_{N,w,h_N'}'(\lam)=0$ exist in $(\lam_i^N,\lam_i^N)$ for $N$ sufficiently large.

The critical points on the region $(-\infty,\lam_1]$ have a different characterization, following from instead applying the lemma to the convex function $(R_{N,w}(\lam)me_1,me_1)$ and tracking infimizers. By the lemma, there always exists a unique $\nu_{1}^N\in (-\infty, \lam_1^N]$, such that $J_{N,w,h_N'}'(\nu_1^N)=0$, and similarly a point $\nu_1$ with the same properties for $\cal{J}_{w,h}$, and such that $\nu_1^N\to \nu_1$.\\
Now that we have: $V_{N,w,h_N'}^k=\{\nu_1^N,\nu_{1;1}^N,\nu_{1;2}^N,\dots \nu_{k;1}^N\}$ and $V_{w,h_N'}^k=\{\nu_1,\nu_{1;1},\nu_{1;2},\dots \nu_{k,1}\}$. Thus in light of the above convergences, we have proven the theorem.
\end{proof}
\begin{remark}
We note that in view of Theorem \ref{duality-theorem}, $V_{N,w,h'}^0$ consists of at least one point, and at most two. One point always lies in $(-\infty, \lam_1^N]$, and the other in $(\lam_1^N,\lam_2^N]$. The monotonicity of Theorem \ref{duality-theorem} implies that the smaller $J_{N,w,h'}$ value is attained on the point in $(-\infty, \lam_1^N]$. In view of Remark \ref{strictremark}, we see that the same is true of $\cal{J}_{w,h}$. As the critical point on $(-\infty,\lam_1]$ is a supremizer on its domain by concavity, this shows that :\[\inf \Lam_{w,h}^0=\sup_{\lam<\lam_1^w}\cal{J}_{w,h}(\lam).\label{Supformula}\]
\end{remark}

\subsection{Discrete Quasi-Derivatives \label{quasi-section}}
In this subsection, and the next, we assume we have passed to a subsequence such that the convergence of Section \ref{Deterministic} holds. We realize an expression of the above matrix ensembles in a form more similar to our construction of $\cal{H}_{}$ as a Sturm-Liouville operator. In particular, we introduce a notion of discrete quasi-derivative related to the above construction of $H_{N,w}$. We finish this subsection with a convergence theorem based on these quasi-derivatives.

We define the following operators on $\R^N$:
\[D_N^{[1]}=D_N-(y_1^N)_{\times}-(y_2^N)_{\times}\frac{1}{2}(T_N+T_N^*);D_N^{[2]}=D_ND_N^{[1]}.\]
The relation of these quasi-derivatives to the discrete operators is slightly more subtle than in the continuum case. Heuristically, the term $(H_{N,w}v)_1=wmv_1-m(Dv)_1$ may be thought to be weakly enforcing the $w$-Robinson boundary condition, while the remaining terms of the operator are independent of $w$, and form a discretization of the maximal operator. More formally, we introduce the notation $H_{N}=T^*_N H_{N,w}$ for any choice of $w$. We note that as only the first term of $H_{N,w}$ depends on $w$, this notation is not abusive. We also note the following identity:
\[H_{N}v=-D_N^{[2]}v-[y_1^N+y_2^N2^{-1}(T_N+T_N^*)]D_Nv.\]
This is in a discrete analogue of (\ref{EGNToperatordef}).

We introduce the following scalar-product from \cite{BV}:
\[\|f\|_{*}^2=\|\sq{1+\bar{\eta}}f\|^2+\|D_Nf\|^2.\]
The following discrete-to-continuous convergence lemma should be compared to Lemma 2.15 of \cite{BV}. It will be crucial in the proof of Proposition \ref{MainConvergenceTheorem}.

\begin{lem} 
Let $f_{N}\in \R^N$ be such that $\|H_{N}f_{N}\|$ and $\|f_N\|_*$ are all uniformly bounded. Then there exists $f\in \cal{D}^{\max}$ such that, up to a subsequence, the following convergences hold:\\
$f_N$ converges to $f$ compact-uniformly and in $L^2$, $D_Nf_N$ converges to $f'$ compact-uniformly and weakly in $L^2$, and $H_{N}f_N$ weakly converges to $\cal{H}f$ in $L^2$.\label{MainConvergence}
\end{lem}
Before proceeding, we need the following elementary lemma.
\begin{lem}
Let $f_N\in \R^N$, and let $f\in C^1$. Assume that $f_N$ and $D_Nf_N$ converge to $f$ and $f'$, respectively, locally weakly in $L^2$. Then $f_N$ converges to $f$ compact-uniformly.
\label{integration}
\end{lem}
\begin{proof}
Let $g_N(x)=f_N(0)+\int_0^x D_Nf_N(y)dy$ be the piecewise-linear version of $f_N$. This function coincides with $f_N$ at $i/m_N$ for $0\le i\le N$, and satisfies $g_N'=D_Nf_N$ a.e.. We note thus that we have that $g_N'\to f'$ locally weakly, so that $g_N'$ is locally in $L^2$. Thus by integrating the inequality, $|g_N(x)-f_N(x)|\le \frac{1}{m_N}|g_N'(x)|$, we see that $f_N-g_N$ locally converges to $0$ in $L^2$, so that $g_N$ converges to $f$ locally-weakly in $L^2$, and thus in $H^1$. By Morrey's Inequality, $H^1(I)$ is a compact subset of $C^{1/3}(I)$ for any compact $I$. As the image of a weakly-convergent sequence in a Banach space is locally strongly-convergent, $g_N\to f$ in $\frac{1}{3}$-H{\"o}lder norm, and thus also in the compact-uniform sense. The same argument implies that $f_N-g_N$ converges compact-uniformly to $0$. Combined, these statements yield the lemma.
\end{proof}

\begin{proof}[Proof of Lemma \ref{MainConvergence}]
By Lemma 2.15 of \cite{BV}, the uniform bound on $\|f_N\|_*$ implies that there exists some $f\in L^*$, and subsequence along which $f_N$ converges to $f$ uniformly on compacts and in $L^2$, and such that $D_Nf_N$ converges weakly to $f'$. Fix $\ell\in \R^*$, such that $f$ is $\ell$-Robinson. By the Banach-Alagou Theorem, the bound on $\|H_{N}f_N\|$, implies that up to passing to a further subsequence, $H_{N}f_N$ converges weakly in $L^2$ to some $g\in L^2$. 

By the proof of Lemma 2.16 of \cite{BV}, for $h\in C^\infty$, with $h$ of compact support in $(0,\infty)$, we have that:
\[\cal{H}_{\ell}(f,h)=\lim_{N\to \infty}(H_{N,\ell_N}f_N,h)\]
where $\ell_N=m_N$ if $\ell=\infty$ and $\ell_N=\ell$ otherwise. 

It is clear that $T_N,T_N^*\to 1$ on $L^2$, so that $T_Nh\to h$ in $L^2$. Thus we see that:
\[\lim_{N\to \infty}(H_{N,\ell_N}f_N,h)=\lim_{N\to \infty}(H_{N,\ell_N}f_N,T_Nh)=\lim_{N\to \infty}(H_{N}f_N,h)=(g,h).\]
Thus we have that:
\[\cal{H}_{\ell}(f,h)=(g,h).\]
Then by Lemma \ref{domainiden}, proven in Section \ref{prelimsection}, we have that $f\in \cal{D}_\ell$, and $\cal{H}f=g$ a.e.. This implies that $H_{N}f_N$ weakly converges to $\cal{H}f$.
	
Now all we have to prove is the compact-uniform convergence of $D_Nf_N$ to $f'$. To do this, we compare the formulae:
\[\cal{H}f=-f^{[2]}-yf' \label{ConvergenceFormula1}\]
\[H_{N}f_N=-D_N^{[2]}f_N-[(y_1^N)_{\times}+(y_2^N)_{\times }\frac{1}{2}(T_N+T_N^*)]D_Nf_N\label{ConvergenceFormula2}.\]
By the weak convergence of $D_Nf$ to $f'$ and the convergence of $T_N,T_N^*\to 1$ on $L^2$, we see that $\frac{1}{2}(T_N+T_N^*)D_Nf_N$ converges weakly to $f'$ in $L^2$. By the compact-uniform convergence of $y_i^N$ to $y_i$ for $i=1,2$, we see that have that $[(y_1^N)_\times+(y_2^N)_\times \frac{1}{2}(T_N+T_N^*))]D_Nf_N$ converges $yf'$ locally-weakly in $L^2$. As the left-hand size of (\ref{ConvergenceFormula2}) converges to (\ref{ConvergenceFormula1}) weakly, we see that $D_{N}^{[2]}f_N$ converges locally weakly to $f^{[2]}$. It is clear from the above results that $D_N^{[1]}f_N$ converges locally weakly to $f^{[1]}$. Thus by Lemma \ref{integration}, we have that $f^{[1]}_N=D_Nf_N-[(y^N_1)_{\times}+(y_2^N)_{\times}\frac{1}{2}(T_N+T_N^*)]f_N$ converges to $f^{[1]}=f'-yf$  in the compact-uniform sense. 

Thus to show that $D_Nf_N$ converges to $f'$ in the compact-uniform sense, we only need to show that $-[(y^N_1)_{\times}+(y_2^N)_{\times}\frac{1}{2}(T_N+T_N^*)]f_N$ converges to $-yf=-y_1f-y_2f$ in the compact-uniform sense. It is clear that $(y^N_1)_{\times}f_N$ converges to $y_1f$, so we only worry about the second term. It is clear that $T_N^*D_Nf_N$ converges compact-uniformly to $f'$ by continuity of $f'$. We show that $(y_2^N)_{\times}T_N^*f_N$ converges to $y_2f$ in the compact-uniform sense. We note that for $x\in \R_+$, we have: \[\begin{split}|(y_2^N)_{\times}&T_N^*f_N)(x)-y_2(x)f(x)|\le\\ |y_2^N(x)||T_N^*f_N(x)-T_N^*f(x)|+|&(y_2^N)(x)-y_2(x)||T_N^*f(x)|+|y_2(x)T_N^*f(x)-y_2(x)f(x)|\end{split}.\]
It is clear that the first and second terms go to zero locally-uniformly in $x$, so it suffices to deal with the third term. This term admits the bound $y_2(x)\int_{x-m_N^{-1}}^{x}|f'(y)|dy$ if $x\ge m_N^{-1}$, and the bound $|y_2(x)f(x)|$ if $x<m_N^{-1}$. As $y_2(0)f(0)=0$, and $y_2f$ is continuous, the supremum of $y_2(x)f(x)$ over $[0,m_N^{-1}]$ goes to zero in $N$. Additionally, $y_2(x)\int_{x-m_N^{-1}}^{x}|f'(y)|dy$ admits a bound by $y_2(x)m_N^{-1}\sup_{\in I}|f'(y)|$ over any compact interval $I$. Combining these bounds establishes the desired compact-uniform convergence. This completes the proof.
\end{proof}

\begin{corr}
If $f_N\in \R^N$, and $\|f_N\|$ and $\|H_{N,\ell_N}f_N\|$ is bounded for any choice of $\ell_N$, then the conclusion of Lemma \ref{MainConvergence} holds.
\label{maincorr}
\end{corr}
\begin{proof}
We verify the hypothesis of Lemma \ref{MainConvergence}.
As $\|H_{N}f_N\|\le \|H_{N,\ell_N}f_N\|$, we only need to verify that $\|f\|_*$ is bounded. We note that as, $(H_{N,\ell_N}f_N,f_N)\le \frac{1}{2}((H_{N,\ell_N}f_N,H_{N,\ell_N}f_N)+(f_N,f_N))$, we have that $(H_{N,\ell_N}f_N,f_N)$ is bounded. Now we recall from Lemma 2.13 of \cite{BV}, that there are constants, $c,C>0$, such that:
\[C\|f_N\|_*^2\le c\|f_N\|^2+\|H_{N,\ell_N}f_N\|^2.\]
This implies that $\|f_N\|_*$ is uniformly bounded as desired.
\end{proof}

\subsection{Proof of Proposition \ref{MainConvergenceTheorem}\label{ConvergenceSubsection}}
As in the previous subsection, we will assume we are in the case of Section \ref{Deterministic} for this subsection. We are concerned here with the proof of Proposition \ref{MainConvergenceTheorem}. This is done by proving a convergence result for the first column of $R_{N,w}(\lam)$, from which we may isolate the first entry. Note that for any $\lam$, there is some $\ell=\ell(\lam)$ and some $\varphi_\lam\in\cal{D}_\ell$, with $\|\varphi_\lam\|^2=1$ such that $\cal{H}_{\ell}\varphi_\lam=\lam \varphi_\lam$. This is immediate from the results of the next section (see Proposition \ref{WeylSolutions}) as for any $w\in \R^*$, either $\lam \in \sigma_w(\cal{H})$, or there is a solution with $w^{\perp}$-Robinson boundary conditions. We have by Theorem \ref{resolvantconvergence} (see also \cite{BV}), that there are thus $(v^N_\lam,\lam_N)\in \R^N\times \R$, such that $\|v^N_\lam\|=1$, $\lam_N\to \lam$,  $v_\lam^N\to \varphi_\lam$ in $L^2$, and:
\[H_{N,\ell_N}v^N_\lam=\lam_Nv^N_\lam. \label{EigenIdentity}\]
Where $\ell_N=\ell$ if $\ell\in \R$ and $\ell_N=m_N$ otherwise. We will fix this notation for the rest of this subsection, and will denote $v_{\lam}=v^N_{\lam}$ when $N$ is clear. The first lemma shows that the convergence of $v_\lam$ to $\varphi_\lam$ is infact very strong.

\begin{lem}
We have that $v^N_\lam$ to $\varphi_\lam$ in the modes of Lemma \ref{MainConvergence}.\label{EigenConvergence}
\end{lem}
\begin{proof}
We have that $(H_{N,\ell_N}v_{\lam},H_{N,\ell_N}v_{\lam})=|\lam_N|^2\|v^n_{\lam}\|^2=|\lam_N|^2$, which is bounded. Thus along any subsequence of $N$, we may find a further subsequence of $N$, such that $v_{\lam}^N$ converges to $\varphi_\lam$ in the modes of Lemma \ref{MainConvergence}. This establishes that $v_\lam^N$ converges to $\varphi_\lam$ in the modes of Lemma \ref{MainConvergence}.
\end{proof}

Now we will relate the $v_{\lam}$ to the problem at hand. We note that (\ref{EigenIdentity}) implies that \[(\ell_N+y_{1;1}^N)mv_1-mD_Nv_1=\lam_N v_1.\]
Now recalling the spike parameter, $w=w_N$, let us assume that $\lam\notin \sigma(H_{N,w})$, then we also have that:
\[v_\lam^N=(H_{N,w}-\lam_N)^{-1}(H_{N,w}-\lam_N)v_\lam=(H_{N,w}-\lam)^{-1}[H_{N,\ell_N}-\lam_N+(w-\ell_N)mE_{11}]v_{\lam}\]
\[(w-\ell_N)(v_{\lam})_1(H_{N,w}-\lam_{N})^{-1}me_1.\]
Combining these observations, we obtain:
\[v_\lam=[w(v_\lam)_1-D_Nv_\lam-m^{-1}\lam_N+y_{1;1}^N]R_{N,w}(\lam_N)me_1\label{DiscreteWeyl}.\]
We will use this observation, combined with our observed convergence of $v_{\lam}^N$, to obtain convergence results for $R_{N,w}(\lam)me_1$. Our first result is for $w$-spiked ensembles with $w\in \R$.
\begin{lem}
Assume $w\in \R$. Then for $\lam\notin \sigma_{w}(\cal{H})$, $R_{N,w}(\lam)me_1$ converges to $\varphi_{\lam}^w$ in the modes of Lemma \ref{MainConvergence}.
\label{ResolvantRow1}
\end{lem}
\begin{proof}
We note that as $\lam\notin \sigma_w(\cal{H})$, we have that $\lam\notin \sigma(H_{N,w})$ for large enough $N$ by Theorem \ref{resolvantconvergence}. Assume we are in such a case for the remainder of the proof.
We note that by Lemma \ref{EigenConvergence}, we have that: \[w_N(v_\lam^N)1-D_N(v^N_\lam)_1-m_N^{-1}\lam_N+y_{1,1}^N\to w\varphi_{\lam}(0)-\varphi_{\lam}'(0).\]
As $\lam\notin \sigma_w(\cal{H})$, the latter quantity is nonzero. Thus we obtain by (\ref{DiscreteWeyl}) that $R_{N,w}(\lam_N)me_1$ converges to $\varphi_\lam/(w\varphi_\lam(0)-\varphi_\lam'(0))$. This is an eigenfunction of $\cal{H}_{}$ of the same eigenvalue and boundary conditions as $\varphi^w_\lam$, and thus are multiples of each other by simplicity of the spectrum (See Lemma 2.7 of \cite{BV}). As no multiple of a function satisfying the $w^{\perp}$-Robinson boundary conditions satisfies $w^{\perp}$-Robinson boundary condition, we see that $\varphi_\lam/(w\varphi_\lam(0)-\varphi_\lam'(0))=\varphi_\lam^w$.

Now we need to show that $R_{N,w}(\lam)me_1$ converges to $\varphi_\lam^w$. To do this, we note the following the application of the first resolvent identity:
\[R_{N,w}(\lam)me_1-R_{N,w}(\lam_N)me_1=(\lam-\lam_N)R_{N,w}(\lam)R_{N,w}(\lam_N)me_1\label{FirstResolvantIdentity}.\]
Denote $u_N=R_{N,w}(\lam_N)me_1$. We note that:
\[\|R_{N,w}(\lam)u_N\|\le \frac{1}{d(\lam,\sigma(H_{N,w}))}\|u_N\|.\]
As $d(\lam,\sigma(H_{N,w}))\to d(\lam, \sigma(\cal{H}_{w}))\neq 0$ by Theorem \ref{resolvantconvergence}, and $\|u_N\|\to \|\varphi_\lam^w\|$, the latter is bounded uniformly in $N$.
We also have that:
\[\begin{split}
\|H_{N,w}R_{N,w}(\lam)u_N\|\le \|\lam R_{N,w}(\lam)u_N\|+\|u_N\|.
\end{split}\]
Both of which are uniformly bounded in $N$. Thus by Corollary \ref{maincorr} and (\ref{FirstResolvantIdentity}), we have that $R_{N,w}(\lam)me_1-R_{N,w}(\lam_N)me_1$ converges to $0$ in the modes of Lemma \ref{MainConvergence}. This and the convergence of $R_{N,w}(\lam_N)me_1$ to $\varphi_\lam^w$ completes our proof.
\end{proof}

\begin{lem}
Assume $w=\infty$. Then for $\lam\notin \sigma_\infty(\cal{H})$, $R_{N,w}(\lam)wme_1$ converges to $\varphi_{\lam}^\infty$ in the modes of Lemma \ref{MainConvergence}.
\label{ResolvantRow2}
\end{lem}
\begin{proof}
As in Lemma \ref{ResolvantRow1}, we may choose $N$ large enough that $(v_\lam^N)_1\neq 0$ and $w_N\neq 0$.
We write:
\[(v_\lam^N)_1=[(v_\lam^N)_1-D_N(v_\lam^N)_1/w-m^{-1}\lam/w+y_{1,1}^N/w](R_{N,w}(\lam)wme_1)_1\]
so that $[v_1-D_nv_1/w-m^{-1}\lam/w+y_{1,1}^n/w]\To \varphi_{\lam}(0)$ by Lemma \ref{EigenConvergence} and the growth of $w_N$. With this modification, the proof of Lemma \ref{ResolvantRow1} works exactly.
\end{proof}

We now  proceed with the proof of Proposition \ref{MainConvergenceTheorem}.

\begin{proof}[Proof of Proposition \ref{MainConvergenceTheorem}]
Let us first assume we are in the case that $w\in \R$.
It suffices to prove that:
\[(R_{N,w}(\lam)me_1,me_1)\To\varphi_{\lam}^{w}(0)\]
compact-uniformly in $\lam\in \R-\sigma_w(\cal{H})$.

We note that $(R_{N,w}(\lam)me_1,me_1)=(R_{N,w}(\lam)me_1)_1$. Thus we have that, pointwise in $\lam$, $(R_{Nw}(\lam)me_1)_1\to \varphi_\lam^w(0)$ by Lemma \ref{ResolvantRow1}. To show that this pointwise convergence is compact-uniform, if suffices to show compact-uniform convergence of the derivatives. This is, we must show compact-uniform convergence of
$\|R_{N,w}(\lam)me_1\|$ to $\|\varphi_\lam^w\|$ in $\lam$ (see Proposition \ref{WeylSolutions}). The pointwise convergence of this sequence follows from Lemma \ref{ResolvantRow1}. By the Arzel{\`a}-Ascoli Theorem, to show compact-uniform convergence, it suffices to establish equicontinuity of the family $R_{N,w}(\lam)me_1)$ on compact subsets of $\R-\sigma_w(\cal{H})$. For this, note that:
\[|(\|R_{N,w}(\zeta)me_1\|-\|R_{N,w}(\lam)me_1\|)|\le \|R_{N,w}(\zeta)me_1-R_{N,w}(\lam)me_1\|.\]
Thus by applying (\ref{FirstResolvantIdentity}), we have:
\[\|R_{N,w}(\zeta)me_1-R_{N,w}(\lam)me_1\|\le \frac{|\zeta-\lam|}{d(\zeta,\sigma(H_{N,w}))}\|R_{N,w}(\lam)me_1\|.\]
We have that $d(\zeta,\sigma(H_{N,w}))\to d(\zeta,\sigma(\cal{H}_{w}))$ compact-uniformly in $\zeta$ by Theorem \ref{resolvantconvergence}. This and the pointwise convergence of $\|R_{N,w}(\lam)me_1\|$ establishes equicontinuity. This concludes the proof of the case $w\in \R$.\\

Now assume that $w=\infty$. It suffices to prove that: \[(R_{N,w}(\lam)wme_1,wme_1)-w\To(\varphi_{\lam}^{\infty})'(0)\]
compact-uniformly in $\lam\in \R-\sigma_\infty(\cal{H})$.

We first show pointwise convergence. As $(v_\lam^N)_1\to \varphi_\lam(0)\neq 0$, we have that $(v_\lam^N)_1\neq 0$ for large enough $N$, which we will henceforth assume. Thus, we note that by (\ref{DiscreteWeyl}), we have that:
\[\begin{split}(R_{N,w}(\lam)&wme_1,wme_1)=\\ w(R_{N,w}(\lam)wme_1)_1=w+(v_\lam^N)^{-1}_1[&(D_Nv_\lam^N)_1+m^{-1}\lam-y_{1;1}^N](R_{N,w}(\lam)wme_1)_1\end{split}.\]
We also note that
\[(v_\lam^N)^{-1}_1[(D_Nv_\lam^N)_1+m^{-1}\lam-y_{1;1}^N](R_{N,w}(\lam)wme_1)_1\to \varphi_\lam(0)^{-1}\varphi_\lam'(0)=(\varphi_\lam^{\infty})'(0).\]
Together these establish the desired pointwise convergence. The proof of uniform convergence now proceeds identically to the case of $w\in \R$.
\end{proof}

\section{Preliminaries on Stochastic Schrodinger Operators \label{prelimsection}}

In this section, we will establish some technical results used in Section \ref{resultsection}. We start by recalling and reformulating the stochastic Schrodinger operators introduced in \cite{RRV}, and additionally studied in \cite{BV}. These operators are heuristically of the form
\[\cal{H}=-\frac{d^2}{dx^2}+y'\label{operatordef}\]
with $y\in C^0$, where $y$ be taken to satisfy $y(0)=0$ and some growth conditions (See (\ref{growth1}) and (\ref{growth2})). These operators occur as "continuum limits" of families of tridiagonal matrix ensembles (see \cite{RRV} for a rigorous statement). We will show that the $\cal{H}$ admit an description as a self-adjoint operator on $L^2$, whose eigenpairs coincide with that of \cite{BV,RRV}, and that they admit an specific family of eigenfunctions used extensively in Section \ref{resultsection} (see Proposition \ref{WeylSolutions}).

The case of $y=\frac{1}{2}x^2+\frac{\sq{2}}{\beta}B_x$, where $B_x$ is a standard Brownian motion, is the $\beta$-stochastic Airy operator introduced above. See \cite{BV,RRV} for the basic properties of the eigenvalue problem of this operator, and it's relation to edge statistics of the $\beta$-Hermite ensemble. In this case our main result was proven by \cite{Mi}, whose work we build on.

\subsection{Definition of $\cal{H}$ as a Sturm-Liouville Operator \label{quasiderivative}}

In this subsection we will review the definition of $\cal{H}$ in the framework of \cite{EGNT}. In particular, we will review the relevant quasi-derivatives, are crucial to both the definition of the domain of the operator, and also our analysis of the discrete-to-continuous convergence.

We will first show how the heuristic formula (\ref{operatordef}) fits into the scheme of \cite{EGNT}. First we note the formal identities:
\[\cal{H}=-\frac{d}{dx}(\frac{d}{dx}-y)-y\frac{d}{dx}=-\frac{d}{dx}(\frac{d}{dx}-y)-y(\frac{d}{dx}-y)-y^2.\]
In particular, with the following notation \[\frac{d^{[1]}}{dx^{[1]}}=\frac{d}{dx}-y;\quad \frac{d^{[2]}}{dx^{[2]}}=\frac{d}{dx}\frac{d^{[1]}}{dx^{[1]}}\label{quasider}\]
we have that
\[\cal{H}=-\frac{d^{[2]}}{dx^{[2]}}-y\frac{d}{dx}=-\frac{d^{[2]}}{dx^{[2]}}-y\frac{d^{[1]}}{dx^{[1]}}-y^2 \label{EGNToperatordef}.\]
The rightmost side of (\ref{EGNToperatordef}) is of the form of 1.1 of (\cite{EGNT}). The quantities $\frac{d^{[i]}}{dx^{[i]}}$ are the "quasi-derivatives" of the problem as in \cite{EGNT}. We will denote $f^{[i]}:=\frac{d^{[i]}}{dx^{[i]}}f$. We define the following domains on which $\cal{H}$ act:
\begin{defin}
\[\cal{D}_{\max}:=\{f\in H^1_{\loc}\cap L^2:f^{[1]}\in H^1_{\loc}, \cal{H}f\in L^2\}\]
\[\cal{D}_w:=\{f\in \cal{D}_{\max}:wf(0)=f'(0)\};\; w\in \R\]
\[\cal{D}_{\infty}:=\{f\in \cal{D}_{\max}:f(0)=0\}\]
\[\cal{D}_{\max}^{\loc}:=\{f\in H^1_{\loc}:f^{[1]}\in H^1_{\loc}, \cal{H}f\in L^2_{\loc}\}\] and similarly for for $\cal{D}_w^{\loc}$.
\end{defin}

The significance of these subspaces is that the various $\cal{D}_w$ will serve as the various domains of self-adjointness for the operator $\cal{H}$. The proof of this fact is postponed to the next subsection. We will use $\cal{H}$ without mention of boundary conditions to refer to the operator considered on $\cal{D}_{\max}$. We will notate $\cal{H}_{w}:=\cal{H}|_{\cal{D}_w}$ when it is appropriate. We note that $\cal{H}$ coincides with the "Maximal Operator" of (\ref{EGNToperatordef}), and so in particular it is closed by Theorem 3.4 of \cite{EGNT}. We will abuse notation and denote by $\cal{H}$ the linear functional on $\cal{D}_{\max}^{\loc}$ and similarly for $\cal{D}_{w}^{\loc}$.

We will for $w\in \R$, refer to the condition $wf(0)=f'(0)$ as the $w$-Robinson condition. We will also refer to the condition $f(0)=0$ as the $\infty$-Robinson condition.

\subsection{Self-Adjointness of $\cal{H}_{w}$ and The Distributional Eigenvalue Problem}

Having taken our definition of $\cal{H}$ as an operator, we show that the eigenvalue problem for $\cal{H}_{w}$ coincides with the eigenproblem of \cite{BV}, which we recall below. In the proof of this, we will also show that $\cal{H}_{w}$ is self-adjoint.

Assume that there is $\eta\in L^1_{\loc}$, and $\omega\in C^0$ with $\eta(0)=\omega(0)=0$ such that we have: \[y(x)=\int_0^x\eta(z)dz+\omega(x)\label{decomposition}.\]
Moreover, we assume that there exists unbounded, non-decreasing, continuous functions $\bar{\eta}(x)>0$, $\zeta(x)\ge 1$, as well as a constant $\kappa\ge 1$, such that:
\[\bar{\eta}(x)/\kappa-\kappa\le \eta(x)\le \kappa(\bar{\eta}(x)+1)\label{growth1}\]
\[|w(x)-w(\xi)|^2\le \kappa(1+\bar{\eta}(x)/\zeta(x))\label{growth2}\]
for all $x,\xi\in\R$ with $|x-\xi|\le 1$. We define the weighted-Sobolev norm: \[\|f\|_{*}^2=\|f\sq{(1+\bar{\eta})}\|^2+\|f'\|^2\] and denote the corresponding Hilbert Spaces as \[L^*=\{f\in H^1:\|f\|_*<\infty\}.\]
We will denote $L^*_w=L^*$ for $w\in \R$, and $L^*_\infty=\{f\in L^*:f(0)=0\}$. Similarly we will denote $C_w^{\infty}=C^{\infty}$ for $w\in \R$, and $C_\infty^\infty=\{f\in C^\infty:f(0)=0\}$. We now define, for $w\in \R^*$ and $f,g\in C_w^{\infty}$:
\[\cal{H}_{w}(f,g)=(f',g')-((f g)',y)+wf(0)g(0)\]
where the last term is omitted if $w=\infty$.
We recall the key properties of this bilinear form (established as Fact 2.1 and Lemma 2.3 of \cite{BV} respectively)

\begin{note}
Every $L^*_w$-bounded sequence has a subsequence converging in all the following modes: weakly in $L^*$, compact-uniformly, and in $L^2$.
\label{BVconvergence}
\end{note}

\begin{note}
For each $w\in \R^*$, we have a unique, continuous, symmetric extension of $\cal{H}_{w}$ to $(L^*_w)^2$. Moreover, we have constants $c,C>0$ such that:\[c|f|^2_*-C|f|^2\le \cal{H}_{w}(f,f)\le C|f|^2_*.\]
\label{BVinequality}
\end{note}
We see from these that $\cal{H}_{w}$ is lower semi-bounded, closed, and completely-continuous with respect to $L^2_w$. By polarization, such bilinear form gives an unbounded operator on $L^*_w$. In particular by Theorem VIII.15 of \cite{ReedSimon1} we have that there is such a bounded-below, self-adjoint unbounded operator $\overline{\cal{H}}_{w}$ giving this bilinear form. The domain of this operator, denoted by $\overline{\cal{D}}_w$, consists of functions, $f\in L^*_w$, such that there is $g\in L^2$, with $\overline{\cal{H}}_{w}(f,h)=(g,h)$ for all $h\in L^*_w$ (or equivalently, $h\in C_w^{\infty}$). The operator is defined on this domain by $\overline{\cal{H}}_{w}(f)=g$. The assumption of complete continuity implies that the spectrum of $\overline{\cal{H}}_{w}$ has pure-point spectrum by Theorem XIII.64 of \cite{ReedSimon4}. We have the following observation:
\begin{note}
We have that the distributional eigenvalue problem for $\cal{H}_{w}$, as in Definition \textup{2.4} of \textup{\cite{BV}}, coincide with the eigenvalue problem of $\overline{\cal{H}}_{w}$. In particular, they both may be defined as pairs $(f,\lam)\in L_w^*\times \C$, such that for any $h\in C_w^\infty$ we have:\[\cal{H}_{w}(f,h)=\lam(f,h).\]
We note as a corollary of this identification, $\overline{\cal{H}}_{w}$ has a simple spectrum (See Lemma \textup{2.7} of \textup{\cite{BV}}).
\end{note}
We now need to relate the operator $\overline{\cal{H}}_{w}$ to the Sturm-Liouville operator $\cal{H}_{w}$ defined above. This is done in the following proposition:
\begin{prop}
For $w\in \R^*$, we have that $\cal{H}_{w}=\overline{\cal{H}}_{w}$. \\
Additionally, for $f\in \cal{D}^{\max}$ and $g\in L^*_w$ we have that:
\[\cal{H}_{w}(f,g)=(\cal{H}_{}f,g)+(f'(0)-wf(0))g(0);\ \  w\in \R.\]
While for $w=\infty$, and $f(0)=0$, we have: $\cal{H}_{\infty}(f,g)=(\cal{H}_{}f,g).$
\label{idenificationprop}
\end{prop}
The latter statements of Proposition \ref{idenificationprop} follow easily from integration by parts, so we focus on proving that $\cal{H}_{w}=\overline{\cal{H}}_{w}$, which constitute the remainder of this subsection.

We begin by proving that $\overline{\cal{H}}_{w}\subseteq \cal{H}_{w}$, which follow from the following slightly stronger result:
\begin{lem}
For $w\in \R^*$, assume we have $f\in L^*$ satisfying $w$-Robinson boundary conditions, and $g\in L^2$ such that $\cal{H}_{w}(f,h)=(g,h)$ for all $h\in C^\infty$, such that $h$ has compact support in $(0,\infty)$. Then we have $f\in \cal{D}_w$ and $\cal{H}_{w}f=g$.
\label{domainiden}
\end{lem}
Before proving this we recall the following useful criterion, which occurs as Lemma 2 of \cite{Mi}.
\begin{lem}\textup{\cite{Mi}}
Let $\alpha\in L^\infty_\loc$, and let $\gamma\in L_{\loc}^1$. Then if
\[\int_0^\infty h(t)\alpha(t)dt=\int_0^{\infty}h'(t)\gamma(t)dt\]
for all $h\in C^\infty$, with support compactly contained in $(0,\infty)$, then $\alpha\in H^1_{\loc}$ and $-\alpha'=\gamma$ holds a.e.
\end{lem}

\begin{note}
This statement slightly differs from that of Lemma \textup{2} of \textup{\cite{Mi}}, but follows immediately from the proof.
\end{note}
\begin{proof}[Proof of Lemma \ref{domainiden}]
For $h\in C^\infty$ of compact support in $(0,\infty)$, we note that $\cal{H}_{w}(f,h)=(g,h)$ may be rewritten as:
\[\int_0^\infty [f^{[1]}(t)h'(t))]dt=\int_0^\infty [g(t)h(t)-y(t)f'(t)h(t)]dt.\]
Thus by the preceding lemma we have that $f^{[1]}\in H^1_{\loc}$ and $-f^{[2]}=y(t)f'+g(t)$ a.e., or equivalently, $\cal{H}_{}f=g$.
\end{proof}

Now to show the reverse inclusion, it suffices to show that $\cal{H}_{w}$ is self-adjoint, as we may apply the adjoint map to the previous inclusion. To do this, we will employ the results and methods of Sturm-Liouville Theory.

To do this we will make use of the notation of an operator being Limit Circle (l.c) or Limit Point (l.p) at a boundary point (see pg.11 of \cite{EGNT}). In particular by Theorem 6.2 of \cite{EGNT}, to demonstrate that $\cal{H}_{w}$ is self-adjoint, it suffices to show that the operator is l.c at zero and l.p at $\infty$. The regularity of $y$ at $0$ implies that the operator is l.c at zero (See Theorem 4.1 of \cite{EGNT}), so we are left with proving that it is l.p at $\infty$. We recall two classical results from the theory of Sturm-Liouville operators:

\begin{theorem}
\textup{(Theorem 11.7, \cite{EGNT})} Let $L$ be Sturm-Liouville operator on $[0,\infty)$, that is regular at $0$. If any eigenfunction of a Sturm-Liouville operator has a finite number of zeros, then $L$ is l.p at $\infty$.
\end{theorem}

\begin{theorem}
\textup{(Theorem 11.13,\cite{EGNT})} Let $L$ a Sturm-Liouville operator that is bounded-below on compactly-supported functions in its domain. Then there is $\alpha$, such that $\lam<\alpha$, $L-\lam$ possess solutions that have finitely many zeros.
\end{theorem}

We define $\cal{D}_{c}=\{f\in \cal{D}_{\max}:\supp(f)\text{ is compact in }(0,\infty)\}$. We write $\cal{H}'_{}$ for the restriction of $\cal{H}$ to $\cal{D}_c$. We see that in view of the previous two theorems, it suffices to prove that $\cal{H}'$ is lower-bounded. To do so, if suffices to show that $\cal{D}_c\subseteq \overline{\cal{D}_0}$, as $\overline{\cal{H}}_{0}$ is lower-bounded.

Now let $T>0$, and consider \[\cal{D}^T=\{f\in H^1([0,T])\cap L^2([0,T]):f^{[1]}\in H^1([0,T]),\cal{H}f\in L^2([0,T]),f'(0)=f(T)=0\}.\]
Let $\cal{H}_{}^T$ be the operator given by (\ref{EGNToperatordef}) on $\cal{D}^T$. As $y$ is regular around $T$ (See Theorem 4.1 of \cite{EGNT}), we see that $\cal{H}$ is l.p at $T$, so that $\cal{H}^T$ is self-adjoint. Considering the restriction, $\overline{\cal{H}}^T$, of $\overline{\cal{H}}_{0}$ to functions with support in $[0,T]$, we see that $\overline{H}_{}^T\subseteq \cal{H}^T$, and so by self-adjointness of both of these operators, $\overline{\cal{H}}_{}^T= \cal{H}^T$. Thus $\cal{D}^T\subseteq \overline{\cal{D}_0}$. But $\cal{D}_c=\bigcup_{T>0}\cal{D}^T$, so we have that $\cal{D}_c\subseteq \overline{\cal{D}_0}$. This completes the proof of Proposition \ref{idenificationprop}.

\subsection{Weyl Solutions for $\cal{H}$}

The remainder of this section will be spent defining families of eigenfunctions that are important to our analysis. These families of eigenfunctions are quite old, originally appearing in the foundational works of Weyl-Titchmarsh Theory.

For $w\in \R$, we define the $w^{\perp}$-Robinson condition to be: $wf(0)+1=f'(0)$. This is chosen so that when $g$ satisfies the $w$-Robinson Boundary Conditions with $g(0)=1$, and $f$ satisfies the $w^{\perp}$-Robinson Condition, we have $W(g,f)=1$, where here \[W(g,f)=g(0)f'(0)-g'(0)f(0)\] denotes the Wronskian. For $w=\infty$, we will refer to the condition $f(0)=1$ as the $\infty^{\perp}$-Robinson condition. It is chosen so that $W(g,f)=1$ when $g$ is $\infty$-Robinson and $g'(0)=1$. We recall that a function, $f:\R\to \R$, is called real meromorphic, if it is the restriction of a meromorphic function. Let us denote $\sigma_w(\cal{H}_{})=\sigma(\cal{H}_{w})$. We have the following result:

\begin{prop}
For each $w\in \R^*$, there is a unique real meromorphic family in $\lam$ of $w^{\perp}$-Robinson functions $\varphi_\lam^w\in \cal{D}^{\max}$, which solve the equation $\cal{H}\varphi_{\lambda}^w=\lam \varphi_{\lam}^w$. This family has simple poles precisely at $\sigma_w(\cal{H})$. Moreover we have:  \[\partial_{\lam}\varphi_{\lam}^{w}(0)=\|\varphi_{\lam}^w\|^2;\quad w\in\R\]
\[\partial_{\lam}(\varphi_{\lam}^{\infty})'(0)=\|\varphi_{\lam}^\infty\|^2.\]
\label{WeylSolutions}
\end{prop}
\begin{proof}
As our operator is l.c. at $0$, by the proof of Theorem 8.4 of \cite{EGNT}, there exists real entire families of solutions $\varphi_\lam^w,\theta_\lam^w\in \cal{D}^{\max}_{\loc}$ satisfying the following conditions:
\[\phi_\lam^w(0)=1,(\phi_\lam^w)'(0)=w;\;\theta_\lam^w(0)=0,(\theta_\lam^w)'(0)=1; w\in \R\label{diffidentity1}\] 
\[\phi_\lam^\infty(0)=0,(\phi_\lam^\infty)'(0)=1;\;\theta_\lam^\infty(0)=1,(\theta_\lam^\infty)'(0)=0\label{diffidentity2}.\]	
We note that these satisfy, for any $w\in \R^*$:
\[W(\phi_\lam^w,\phi_\zeta^w)=W(\theta_\lam^w,\theta_\zeta^w)=0;W(\theta_\lam^w,\phi_\zeta^w)=1.\]
In particular, we are in the case of Hypothesis 8.1 of \cite{EGNT}. Thus by Theorem 8.2 (and the preceding discussion) for each $w\in \R^*$, there exists a real meromorphic function $m_w$, with simple poles at $\sigma_w(\cal{H})$, such that:
\[m_w(\lam)\phi_\lam^w+\theta^w_\lam\]
is a real meromorphic family of $L^2$-eigenfunctions, with simple poles at $\sigma_w(\cal{H})$, with $w^\perp$-Robinson initial Condition. We may take this function as our definition of $\varphi_\lam^w$.

Now we prove the differential equalities.  We see that for $w\in\R$, we have $m_w(\lam)=\varphi_\lam^w(0)$ and $m_\infty(\lam)=(\varphi_\lam^\infty)'(0)$. This transforms the differential equalities into $\partial_\lam m_w(0)=\|\varphi_\lam^w\|^2$. We note that as $m_w$ is real on $\R$, we have that: \[\partial_\lam m_w(\lam)=\lim_{\epsilon\to 0}\frac{\Im(m_w(\lam+i\epsilon))}{\epsilon}=\|\varphi_\lam^w\|^2\]
where the final equality follows from Corollary 8.5 of \cite{EGNT}.
\end{proof}

\begin{remark}
As remarked in the introduction, the function $\varphi_\lam^w$ is in-fact quite classical. Namely, $\varphi_\lam^w$ is known as the Weyl solution (corresponding to $w$-Robinson boundary conditions), with $m_w(\lam)$ being the Weyl-Titchmarsh $m$-function. This function occupies a distinguished role in the theory of Sturm-Liouville operators, namely as the Stieltjes transform of the spectral measure (See Lemma \textup{9.1} of \textup{\cite{EGNT}}).
\end{remark}

\begin{remark}
We remark on the following distributional identities
\[(\cal{H}_{w}-\lam)\varphi_{\lam}^{w}=\delta;\quad w\in \R\]
\[(\cal{H}_{\infty}-\lam)\varphi_{\lam}^{\infty}=\delta'\]
These identities, while distributional, serve as an important motivation for the study of such functions. Let us denote $R_{w}(\lam):=(\cal{H}_{w}-\lam)^{-1}$. Then we see that we have that $R_{w}(\lam)\delta=\varphi_{\lam}^w$ and $R_{\infty}(\lam)\delta'=\varphi_{\lam}^{\infty}$, again interpreted in the distributional sense.\\ This interpretation of the resolvent term is the foundation of our convergence analysis for this term in the Langragian-Dual problem. Furthermore, one may look at (\ref{diffidentity1}) and (\ref{diffidentity2}) as heuristically following from differentiating these resolvent identities.
\end{remark}

\newpage

\bibliographystyle{apa}

\end{document}